\documentclass[reqno]{amsart} 
\usepackage{amsmath,amssymb,amsfonts}
\usepackage{mathtools}
\usepackage{graphicx}
\usepackage{color}
\makeatletter
\let\cl@chapter\undefined
\makeatletter
\usepackage[colorlinks=true,pagebackref=true,citecolor = cyan,linkcolor=cyan]{hyperref}
\usepackage[capitalize]{cleveref}
\usepackage{algorithm,algorithmicx,algpseudocode}
\usepackage{tikz}
\usepackage{pgfplotstable}
\usetikzlibrary{arrows,positioning,plotmarks,external,patterns,angles,
decorations.pathmorphing,backgrounds,fit,shapes,calc}
\usepackage{autonum}

\crefname{figure}{Figure}{Figures}

\newcommand{\NN}{\mathbb{N}}
\newcommand{\RR}{\mathbb{R}}

\newcommand{\rd}{\mathrm{d}}
\newcommand{\re}{\mathrm{e}}
\newcommand{\cL}{\mathcal{L}}
\newcommand{\cE}{\mathcal{E}}

\newcommand{\etaLB}{\eta_{\mathrm{LB}}}

\newcommand{\pdmat}{D}
\newcommand{\retainlabel}[1]{\label{#1}\sbox0{\ref{#1}}}
\DeclarePairedDelimiter\paren{\lparen}{\rparen}
\DeclarePairedDelimiterX{\inprd}[2]{\langle}{\rangle}{{#1},{#2}}
\DeclarePairedDelimiterX{\setI}[2]{\{}{\}}{\,{#1}\ \delimsize| \ {#2}\,}
\DeclarePairedDelimiter{\setE}{\{}{\}}

\DeclarePairedDelimiter{\norm}{\|}{\|}
\DeclareMathOperator*{\argmin}{arg\,min}

\DeclareMathOperator{\rqmin}{\omega_{\min}}

\usepackage{tikz}
\usetikzlibrary{positioning}
\usetikzlibrary{calc}
\usetikzlibrary{arrows}
\graphicspath{{fig/}{../fig/}}

\newtheorem{theorem}{Theorem}[section]
\newtheorem{lemma}[theorem]{Lemma}

\newtheorem{proposition}[theorem]{Proposition}

\theoremstyle{definition}

\theoremstyle{remark}
\newtheorem{remark}[theorem]{Remark}

\begin{document}

\title[Existence of Lagrange multiplier approach for gradient flows]{Existence results on\\ Lagrange multiplier approach for gradient flows\\ and application to optimization}

\author{Kenya Onuma}
\address{Department of Mathematical Informatics,
             Graduate School of Information Science and Technology,
             The University of Tokyo,
             Hongo 7-3-1, Bunkyo-ku, Tokyo, 113-0033, Japan}
\author{Shun Sato}
\address{Department of Mathematical Informatics,
             Graduate School of Information Science and Technology,
             The University of Tokyo,
             Hongo 7-3-1, Bunkyo-ku, Tokyo, 113-0033, Japan}
\email{shun@mist.i.u-tokyo.ac.jp}

\begin{abstract}
This paper deals with the geometric numerical integration of gradient flow and its application to optimization. 
Gradient flows often appear as model equations of various physical phenomena, and their dissipation laws are essential. 
Therefore, dissipative numerical methods, which are numerical methods replicating the dissipation law, have been studied in the literature. 
Recently, Cheng, Liu, and Shen proposed a novel dissipative method, the Lagrange multiplier approach, for gradient flows, which is computationally cheaper than existing dissipative methods. 
Although their efficacy is numerically confirmed in existing studies, the existence results of the Lagrange multiplier approach are not known in the literature. 
In this paper, we establish some existence results. 
We prove the existence of the solution under a relatively mild assumption. 
In addition, by restricting ourselves to a special case,  we show some existence and uniqueness results with concrete bounds. 
As gradient flows also appear in optimization, 
we further apply the latter results to optimization problems.
\end{abstract}

\maketitle

\section{Introduction}
\label{sec:intro}

In this paper, we consider the numerical integration of the gradient flow
\begin{equation}\label{eq:GF}
	\dot{x} = - \pdmat \nabla V(x),\quad x(0) = x_0,
\end{equation}
where $ x_0 \in \RR^n $ is an initial condition, $ V : \RR^n \to \RR $ is a differentiable function, 
and the matrix $ \pdmat \in \RR^{n \times n}$ is positive definite but not necessarily symmetric. 

Gradient flows~\eqref{eq:GF} are important class of ordinary differential equations (ODEs) that describe various physical phenomena. 
Consequently, numerical methods for gradient flow~\eqref{eq:GF} have also been intensively studied. 
In particular, specialized numerical schemes that replicate the dissipation law $ \frac{\rd}{\rd t} V(x(t)) \le 0  $ have been devised and investigated. 
Techniques devising and analyzing such a specialized numerical scheme replicating a geometric property of ODEs are known as ``geometric numerical integration'' techniques (cf.~\cite{HLW2010}). 

The discrete gradient method~\cite{G1996} (see also~\cite{MQR1998}) is the most popular specialized numerical method for gradient flows. 
Schemes based on the discrete gradient method are often superior to general-purpose methods, particularly for numerically difficult differential equations. 

Although these schemes allow us to employ a larger step size than general-purpose methods, 
they are usually more expensive per step. 
Most of these schemes require solving an $n$-dimensional nonlinear equation per step. 

Consequently, several techniques to enhance the computational efficiency have been studied in the literature (see, e.g., \cite{KS2022BIT} and references therein). 
In particular, Cheng, Liu, and Shen~\cite{CLSJ2020} proposed the Lagrange multiplier approach,
which replicates the dissipation law. 
Their proposed method is based on splitting the function $V$ in the form
\begin{equation}\label{eq:splitting}
    V(x) = \frac{1}{2}\inprd*{x}{Qx} + E(x),
\end{equation}
where $ Q \in \RR^{n\times n} $ is symmetric, 
and $ E:\RR^n \to \RR $ is a differentiable function. 
Note that the splitting is not unique ($ E $ may contain a quadratic term); 
however, when we consider physical problems, the function $V$ often includes a quadratic term so that we can naturally obtain a splitting (see, e.g., \cite{SXY2019SIREV}). 

Then, we can construct a numerical scheme preserving the dissipation law
by using the implicit midpoint rule for the quadratic term and special treatment for the nonlinear term $ E $, respectively (see \cref{pre:LM} for details). 
The resulting scheme requires solving a scalar nonlinear equation (and $n$-dimensional linear equations) per step, which is quite cheap. 

Unfortunately, however, the existence of a solution of the scalar nonlinear equation is not known in the literature. 
Therefore, in this paper, we establish some existence results. 
First, we prove the existence of the solution under a relatively mild assumption (\cref{existence:general}). 
Second, by restricting ourselves to a case where $ Q $ is the zero matrix, 
we establish several existence and uniqueness results with concrete bounds on the solution (\cref{existence:special}). 

The latter results are useful in the application of the Lagrange multiplier method to optimization problems
because the gradient flow~\eqref{eq:GF} also appears in the context of optimization. 
Investigations on the relationship between optimization methods and the discretization of ordinary differential equations (ODEs) have been reported in the 1980s (e.g., \cite{BB1989,S1988,ZG1990}). 
In addition, inspired by the pioneering work of Su, Boyd, and Cand{\`e}s~\cite{SBC2015} on Nesterov's accelerated gradient method, 
research in this direction has been active again in recent years (see \cite{Wilson2018} and the references therein). 

When considering the optimization, 
the dissipation law is also important: 
it is not merely a guarantee of the monotonic decrease of the function value, 
but can also be used to prove its convergence rate. 
Indeed, the discrete gradient method has recently been applied to optimization problems~\cite{REQS2021,RLS2018,ERRS2018}. 

Because the Lagrange multiplier method is much cheaper than the discrete gradient method per step, 
we consider its application to optimization problems. 
Indeed, when $ Q $ is the zero matrix and $ \pdmat $ is the identity matrix, 
the resulting scheme can be regarded as the well-known steepest descent method, adopting a new step size criterion. 
For the scheme, 
we show the convergence rates for several function classes: (i) general $L$-smooth functions, (ii) convex functions, and (iii) functions that satisfy the Polyak--{\L}ojasiewicz inequality (\cref{sec:cr}). 
We also introduce a relaxation technique to further enhance the computational efficiency (\cref{sec:proposed}). 

\smallskip

It may seem as though the scheme is merely a variant of the basic existing method; 
moreover, as shown in numerical experiments later, 
the actual behavior is almost the same as that of the existing method. 

However, the optimization methods proposed in this paper have the advantage that the relationship between continuous and discrete systems is clear in the proof of convergence rates (see, e.g., \cref{thm:gradrate,thm:LMrate1}). 
In existing research considering the correspondence between continuous and discrete systems, although the discussion on continuous systems is simple, it is often very complicated to prove the corresponding property in discrete systems.
A limitation of this paper is that we deal with the simplest gradient flows; however, 
it suggests that the above issues can be overcome by geometric numerical integration techniques even when we are dealing with more complicated ODEs that appear in optimization. 

The remainder of this paper is organized as follows. 
\Cref{sec:pre} presents the Lagrange multiplier approach and the relation between gradient flow and optimization. 
We show several existence results in \cref{sec:existence}
and convergence rates as an optimization method in \cref{sec:cr}. 
In \cref{sec:proposed}, we introduce a relaxation technique. 
These results are confirmed by numerical experiments in \cref{sec:ne}. 
Finally, \cref{sec:conclusion} concludes this paper. 

\section{Preliminaries}
\label{sec:pre}

\subsection{Lagrange multiplier approach}
\label{pre:LM}



In this section, we review the numerical method proposed by Cheng, Liu, and Shen~\cite{CLSJ2020}, 
which preserves the dissipation law. 

By introducing an auxiliary variable $\eta:\RR_{\geq 0} \to \RR$, we consider the following ODE based on the splitting~\eqref{eq:splitting}:
\begin{subequations}\label{eq:reform}
\begin{align}
		\dot{x} &= - \pdmat \paren*{Qx + \eta\nabla E(x)}, \retainlabel{eq:reform1} \\
		\displaystyle\frac{\rd}{\rd t} E(x) &= \eta\inprd*{\nabla E(x)}{\dot{x}}. \retainlabel{eq:reform2}
\end{align}
\end{subequations}
In view of the chain rule, the auxiliary variable $ \eta $ satisfies $ \eta (t) = 1 $ such that the ODE above is equivalent to the gradient flow~\eqref{eq:GF}. 

Based on the reformulated ODE~\eqref{eq:reform}, 
we consider the following scheme ($x_k \approx x \paren*{kh}$, $\eta_k \approx \eta\paren*{kh}$): 
\begin{subequations}\label{eq:LM}
\begin{align}
		\frac{x_{k+1} - x_{k}}{h} &= - \pdmat \paren*{Q\paren*{\frac{x_{k+1} + x_k}{2}} + \eta_k \nabla E \paren*{x_{k+1/2}^{\ast}}}, \retainlabel{eq:LM1}\\
		E\paren*{x_{k+1}} - E \paren*{x_k} &= \eta_k\inprd*{\nabla E\paren*{x_{k+1/2}^{\ast}}}{x_{k+1} - x_k}. \retainlabel{eq:LM2}
\end{align}
\end{subequations}
Here, $x_{k+1/2}^{\ast}$ is a numerical approximation of $x\paren*{(k+1/2)h}$, 
which can be computed without the unknown variables $ x_{k+1} $ and $ \eta_k $. 
For example, Cheng, Liu, and Shen~\cite{CLSJ2020} employed $x_{k+1/2}^{\ast} := \paren*{3x_k - x_{k-1}}/2$, 
and we employ $ x_{k+1/2}^{\ast} := x_k $ later. 

\begin{remark}\label{rem:DGLM}
The setting $x_{k+1/2}^{\ast} = \paren*{3x_k - x_{k-1}}/2$ is to achieve the second order accuracy, 
whereas the setting $x_{k+1/2}^{\ast} = x_k $ only achieves the first order accuracy. 
Although the former setting is better in terms of accuracy, 
the latter setting is easier to deal with in mathematical analysis. 
In addition, when we employ the latter, 
the Lagrange multiplier approach can be regarded as a special case of the discrete gradient method~\cite{G1996,MQR1998,MQR1999}: 
$ Q\paren*{\frac{x_{k+1} + x_k}{2}} + \eta_k \nabla E \paren*{x_k} $ satisfies the conditions of the discrete gradient. 
\end{remark}

\begin{theorem}[\cite{CLSJ2020}]\label{thm:LMdisp}
A solution $ x_{k+1} $ of the scheme~\eqref{eq:LM} satisfies the discrete dissipation law 
$
	f\paren*{x_{k+1}} \leq f\paren*{x_k}
$.
\end{theorem}

By introducing 
\begin{align}
    p_k &:= \paren*{I+\frac{h}{2} \pdmat Q }^{-1} \paren*{I-\frac{h}{2} \pdmat Q } x_k, &
    q_k &:= \paren*{I+\frac{h}{2} \pdmat Q }^{-1} \pdmat \nabla E \paren*{x_k},
\end{align}
we can rewrite \eqref{eq:LM1} as follows:
$
	x_{k+1} = p_k - h \eta_k q_k
$.
Here, $p_k $ and $q_k$ can be computed by solving the linear equations with the same coefficient matrix $ I+\frac{h}{2} \pdmat Q $, which is invertible for sufficiently small $h$. 
Thus, we can compute $ \eta_k $ by solving 
\begin{equation}\label{eq:lm_nleq}
    F_h (\eta_k; x_k) := E\paren*{p_k - h \eta_k q_k} - E\paren*{x_k} - \eta_k\inprd*{\nabla E\paren*{x_k}}{p_k - h \eta_k q_k - x_k} = 0.
\end{equation}

The scheme~\eqref{eq:LM} requires solving two linear equations with $ n $ variables and a scalar nonlinear equation; moreover, because the coefficient matrix is constant, 
we can solve them quite efficiently. 
However, existence results for the nonlinear equation $F_h (\eta_k; x_k) = 0$ have not been established in the literature.

\subsection{Gradient flow and optimization}
\label{pre:gfopt}

In this section, we consider the unconstrained optimization problem
\begin{equation}\label{eq:opt}
 \min_{ x \in \RR^n} f(x), 
\end{equation}
where the function $f:\RR^n \to \RR$ is assumed to be $L$-smooth (i.e., the gradient is $L$-Lipschitz continuous) and satisfies $ \argmin f \neq \emptyset $. 
Under this assumption, there is an optimal solution $ x^\star $ and an optimal value $ f^\star = f ( x^\star)  $. 
In particular, we consider the relationship between the problem and gradient flow~\eqref{eq:DGI}:
\begin{equation}\label{eq:DGI}
	\dot{x} = - \nabla f(x),\qquad x(0) = x_0,
\end{equation}
which is a special case of the general gradient flow~\eqref{eq:GF}. 


$L$-smooth functions satisfy the following inequalities. 

\begin{lemma}\label{prop:ltan}
    If $f$ is $L$-smooth, the following inequalities hold for all $ x, y \in \RR^n ${\em :} 
    \begin{equation}
    -\frac{L}{2}\norm{y - x}^2 \le f(y) - f(x) - \inprd*{\nabla f(x)}{y-x} \le \frac{L}{2} \norm*{y-x}^2.
    \end{equation}
\end{lemma}

We sometimes assume that the objective function $f$ is convex or satisfies the Polyak--{\L}ojasiewicz (P{\L}) inequality with parameter $\mu>0$ (cf.~\cite{KNS2016}):
\begin{equation}\label{eq:PL}
    \frac{1}{2} \norm*{ \nabla f (x) }^2 \ge \mu \paren*{ f(x) - f^\star }.
\end{equation}


Note that a $ \mu $-strongly convex function satisfies the P{\L} inequality with parameter $\mu$. 
In addition, if a function is $L$-smooth and satisfies the P{\L} inequality with parameter $\mu$, $L\ge \mu $ holds~\cite{GGGM2021}. 

When the objective function $f$ is strictly convex, 
the gradient flow~\eqref{eq:DGI} satisfies the following proposition. 

\begin{proposition}[cf. \cite{HSD2013}]\label{prop:cont}
Let $f$ be a strictly convex function. 
Then, $f$ is a Lyapunov function of the gradient flow~\eqref{eq:GF}, 
and 
\[
	\lim_{t \to \infty} x(t) = x^\star
\]
holds for any initial condition $ x_0 \in \RR^n $. 
\end{proposition}

Based on the above fact, 
some researchers have posited the idea of using a numerical method for gradient flow as an optimization method. 
For example, 
the explicit Euler method
\begin{equation}\label{eq:exEuler}
	\frac{x_{k+1} - x_k}{h} = -\nabla f(x_k),
\end{equation}
coincides with the steepest descent method. 
However, we should carefully choose the step size $h$ to ensure the convergence of this method. 

Here, because the convergence in continuous time, such as in \cref{prop:cont}, is based on the dissipation law, 
the numerical method replicating the dissipation law can be regarded as an optimization method (see, e.g.,~\cite{ERRS2018}). 
Before stepping into the property in discrete systems, 
we review it in continuous systems in this section. 
The gradient flow~\eqref{eq:DGI} satisfies the following three theorems. 

\begin{theorem}\label{thm:gradrate}
The solution $x$ of the gradient flow~\eqref{eq:GF} satisfies
\[
	\min_{0 \leq \tau \leq t} \norm{\nabla f(x(\tau))} \leq \sqrt{\frac{f(x_0) - f^\star}{t}}.
\]
\end{theorem}

\begin{proof}
Since $ f^\star $ is an optimal value, 
\begin{align}
	f(x_0) - f^\star &\geq f(x_0) - f(x(t)) 
					= \int^0_t \inprd*{\nabla f(x(\tau))}{\dot{x}(\tau)} \rd\tau 
					= \int_0^t \norm*{\nabla f(x(\tau))}^2 \rd\tau \\
					&\geq t\min_{0\leq\tau\leq t}\norm*{\nabla f(x(\tau))}^2
\end{align}
holds. 
\end{proof}

\begin{theorem}\label{thm:convrate}
If $f$ is convex, the solution $x$ of the gradient flow~\eqref{eq:DGI} satisfies
\[
	f(x) - f^\star \leq \frac{\norm*{x_0 - x^\star}^2}{2t}.
\]
\end{theorem}

\begin{proof}
Let $ \cE $ be a function defined by  
$
	\cE(t) := t\paren*{f(x) - f^\star} + \frac{1}{2}\norm*{x - x^\star}^2.
$
Then, $ \cE (t) $ decreases along time:
\begin{align}
	\dot{\cE}(t) &= f(x) - f^\star + t\inprd*{\nabla f(x)}{\dot{x}} + \inprd*{x - x^\star}{\dot{x}} \\ 
				&= f(x) - f^\star + \inprd*{x^\star - x}{\nabla f(x)} - t\norm*{\nabla f(x)}^2 \\
				&\leq -t\norm*{\nabla f(x)}^2, 
\end{align}
where the last inequality is due to convexity. 
Therefore, 
\[
	t\paren*{f(x) - f^\star} \le \cE(t) \le \cE(0) = \frac{1}{2}\norm{x_0 - x^\star}^2
\]
holds, which proves the theorem. 
\end{proof}

\begin{theorem}\label{thm:sconvrate}
If $f$ satisfies the Polyak--{\L}ojasiewicz inequality~\eqref{eq:PL} with parameter $\mu > 0$, 
the solution $x$ of the gradient flow~\eqref{eq:DGI} satisfies
\[
	f(x) - f^\star \leq \exp\paren*{-2\mu t}\paren*{f(x_0) - f^\star}.
\]
\end{theorem}

\begin{proof}
Let $ \cL $ be a function defined by $ \cL(t) := f(x) - f^\star $. 
Then, 
\[ \dot{\cL} (t) = \inprd{\nabla f(x)}{\dot{x}} = -\norm{\nabla f(x)}^2 \le - 2 \mu \paren*{f(x) - f^\star} = - 2\mu\cL(t); \]
therefore, $\cL(t) \leq \exp\paren*{-2\mu t}\cL(0) $ holds, which proves the theorem. 
\end{proof}

Ehrhardt, Riis, Ringholm, and Sch\"onlieb~\cite{ERRS2018} showed that 
the discrete gradient method with several known constructions of the discrete gradient satisfies $ f(x_k) - f^\star = O\paren*{1/k} $ for convex functions, and $ f(x_k) - f^\star = O\paren*{\exp(-Ck)} $ for functions satisfying P{\L} inequality~\eqref{eq:PL} ($ C > 0 $ is a constant).

\section{Existence theorems}
\label{sec:existence}

In this section, we establish existence theorems for the Lagrange multiplier method~\eqref{eq:LM} under the assumption $x_{k+1/2}^{\ast} = x_k$. 
First, we establish an existence result for general splitting in \cref{existence:general}. 
Then, we restrict ourselves to a special case and obtain existence results with bounds on the solution $ \eta_k $. 

\subsection{Existence results in general setting}
\label{existence:general}

In this section, we prove the existence of the solution $ \eta $ of the nonlinear equation $ F_h (\eta; x_k) = 0 $ for sufficiently small $h$ by using the intermediate value theorem. 
For this purpose, we first prove that $ F_h (\eta; x_k) > 0 $ holds for sufficiently large $\eta $ (\cref{lem:ex_gen_pos}). 
Then, we prove that there exists $ \eta $ satisfying $ F_h (\eta; x_k) < 0 $ when 
$ \inprd*{ \nabla E \paren*{x_k} }{ \pdmat \nabla V \paren*{x_k} }  \neq 0 $ (\cref{lem:ex_gen_neg}). 
These lemmas imply the desired existence theorem (\cref{thm:ex_gen}) on the case $ \inprd*{ \nabla E \paren*{x_k} }{ \pdmat \nabla V \paren*{x_k} } \neq 0 $. 
In addition, even in the case $ \inprd*{ \nabla E \paren*{x_k} }{ \pdmat \nabla V \paren*{x_k} } = 0 $, by introducing a small perturbation to the splitting, we can return to the case $ \inprd*{ \nabla E \paren*{x_k} }{ \pdmat \nabla V \paren*{x_k} }  \neq 0 $. 

Let us denote the minimum value of the Rayleigh quotient of matrix $A$ by $ \rqmin (A)$. 
Note that, if matrix $A$ is symmetric, $ \rqmin (A) $ coincides with the minimum eigenvalues of $A$. 
Moreover, since the matrix $ \pdmat $ is positive definite, 
$ \rqmin (\pdmat) \ge 0 $ and $ \rqmin \paren*{ \pdmat^{-1} } \ge 0 $ hold. 

\begin{lemma}\label{lem:ex_gen_pos}
Let $ E : \RR^n \to \RR $ be an $L_E$-smooth function. 
Suppose that $ \nabla E (x_k) \neq 0 $ holds and 
$h$ satisfies $ h \paren*{ \rqmin (Q) - L_E } > - 2 \rqmin \paren*{ \pdmat^{-1} }$. 
Then, there exits $ \overline{\eta} \in \RR $ such that $ F_h (\overline{\eta }; x_k) > 0 $ holds.
\end{lemma}

\begin{proof}
Because of the assumption on $h$, the matrix $ I + (h/2) \pdmat Q $ is invertible. Therefore, the assumption $ \nabla E \paren*{x_k} \neq 0 $ implies $ q_k \neq 0 $.

Since $E$ is $L_E$-smooth, we use \cref{prop:ltan} and obtain
\begin{align}
    F_h \paren*{ \eta; x_k }
    &\ge \paren*{ 1 - \eta } \inprd*{ \nabla E \paren*{ x_k } }{ p_k - h \eta q_k - x_k } - \frac{L_E}{2} \norm*{ p_k - h \eta q_k - x_k }^2.
\end{align}
The right-hand side is a quadratic function with respect to $\eta$, whose coefficient of the highest degree is positive:
\begin{align}
    h \inprd*{\nabla E \paren*{ x_k }}{ q_k } - \frac{L_E h^2}{2} \norm*{ q_k }^2
    &= h \inprd*{ \pdmat^{-1} \paren*{ I + \frac{h}{2} \pdmat Q } q_k }{ q_k } - \frac{L_E h^2}{2} \norm*{ q_k }^2 \\
    &\ge h \paren*{ \rqmin \paren*{\pdmat^{-1}} + \frac{h}{2} \rqmin (Q) - \frac{L_E h}{2}  } \norm*{q_k}^2 \\
    &>0.
\end{align}
Therefore, $F_h \paren*{ \eta; x_k } $ is positive for sufficiently large $ \eta$. 
\end{proof}

\begin{lemma}\label{lem:ex_gen_neg}
Suppose that $ \inprd*{ \nabla E \paren*{x_k} }{ \pdmat \nabla V \paren*{x_k} } \neq 0 $ holds. 
Then, there exist $ \underline{\eta} \in \RR $ and $ \overline{h} > 0 $ such that 
$ F_h \paren*{ \underline{\eta} ; x_k } < 0 $ holds for any $ h < \overline{h} $. 
\end{lemma}

\begin{proof}
Since
\begin{equation}
     \frac{\partial p_k}{\partial h} = -\paren*{ I + \frac{h}{2} \pdmat Q }^{-1} \pdmat Q \frac{ p_k + x_k }{2}
\end{equation}
holds, we see
\begin{align}
    \left. \frac{ \partial }{ \partial h } F_h \paren*{ \eta ; x_k } \right|_{h=0}
    &= \inprd*{ \left. \nabla E \paren*{ p_k - h \eta q_k } \right|_{h=0} - \eta \nabla E \paren*{x_k} }{ \left. \frac{\partial}{\partial h} \paren*{ p_k - h \eta q_k } \right|_{h=0} }\\
    &= \paren*{1-\eta} \inprd*{ \nabla E \paren*{x_k} }{ - \pdmat Q x_k - \eta \pdmat \nabla E \paren*{x_k} }.
\end{align}
The right-hand side is a quadratic function with respect to $ \eta $ 
such that its coefficient of the highest degree is positive. 
In addition, under the assumption of the lemma, the quadratic function has two distinct real roots. 
Therefore, for any $ \eta $ between these two real roots, 
$\left. \frac{ \partial }{ \partial h } F_h \paren*{ \eta ; x_k } \right|_{h=0} < 0 $ holds. 
This implies the lemma
because $ F_0 \paren*{\eta; x_k } = 0 $ holds and $ F_h \paren*{ \eta; x_k } $ is continuous with respect to $ h $ for sufficiently small $h$. 
\end{proof}

Combining \cref{lem:ex_gen_pos,lem:ex_gen_neg}, 
we obtain the following existence theorem because of the intermediate value theorem. 

\begin{theorem}\label{thm:ex_gen}
    Let $ E : \RR^n \to \RR $ be an $L_E$-smooth function. 
    If $ x_k $ satisfies $ \inprd*{ \nabla E \paren*{x_k} }{ \pdmat \nabla V \paren*{x_k} } \neq 0 $, 
    then, for all sufficiently small $h$, there exists a solution of the scheme~\eqref{eq:LM}. 
\end{theorem}

\begin{remark}
\label{rem:ex_gen_unique}
The above argument implies that the scalar nonlinear equation $ F_h (\eta;x_k) = 0 $ has at least two solutions. 
In this sense, the usual uniqueness does not hold for the schemes. 
However, roughly speaking, the proof of \cref{lem:ex_gen_neg} implies that 
the two solutions are close to $1$ and $ - \inprd*{\nabla E (x_k)}{\pdmat Q x_k}/\inprd*{\nabla E (x_k)}{\pdmat \nabla E ( x_k) } $ for sufficiently small $h$, respectively. 
Since the solution of the continuous system~\eqref{eq:LM} is $ \eta (t) = 1 $, the solution that is closest to $1$ should be used in numerical computation. 
\end{remark}

Finally, we consider the case $ \inprd*{ \nabla E \paren*{x_k} }{ \pdmat \nabla V \paren*{x_k} } = 0 $. 
If $ \nabla V \paren*{ x_k } = 0 $ holds, $x_k$ is an equilibrium point of the system. 
Therefore, we focus on the case $ \nabla V \paren*{ x_k } \neq 0 $ hereafter. 
In this case, by introducing a small perturbation to the splitting~\eqref{eq:splitting}, we can use \cref{thm:ex_gen}. 

For an arbitrary $ \epsilon \neq 0 $, we consider the splitting 
\begin{align}\label{eq:splitting_per}
    V(x) &= \frac{1}{2}\inprd*{x}{Q_{\epsilon} x} + E_{\epsilon} (x),&
    Q_{\epsilon} &:= (1-\epsilon) Q, &
    E_{\epsilon} (x) &:= E(x) + \frac{\epsilon}{2}\inprd*{x}{Qx}. 
\end{align}
Then, we see
\begin{align}
    \inprd*{ \nabla E_{\epsilon} \paren*{x_k} }{ \pdmat \nabla V \paren*{x_k} }
    &= \inprd*{ \nabla E \paren*{x_k} + \epsilon Q x_k }{ \pdmat \nabla V \paren*{x_k} }\\
    &= \epsilon \inprd*{ Q x_k }{ \pdmat \nabla V \paren*{x_k} }\\
    &= \epsilon \inprd*{ Q x_k + \nabla E \paren*{ x_k } }{ \pdmat \nabla V \paren*{x_k} }, 
\end{align}
where the most right-hand side is nonzero
because $ D $ is positive definite and $ \nabla V(x) = Q x + \nabla E(x) $. 
Therefore, \cref{thm:ex_gen} implies that the Lagrange multiplier scheme with the splitting~\eqref{eq:splitting_per} has a solution. 
Consequently, by using the perturbed scheme only when $ \inprd*{ \nabla E \paren*{x_k} }{ \pdmat \nabla V \paren*{x_k} } = 0 $, 
we can continue to compute numerical solutions.

\subsection{Existence results in a special case}
\label{existence:special}

In this section, we further assume $ \pdmat = I $ and $ Q $ is the zero matrix. 
The results in this section can be extended to the case with general $ \pdmat $ (see \cref{app:ex}); 
however, here we focus on the simple gradient flow~\eqref{eq:DGI} 
because the existence results in this case can be utilized in optimization (see \cref{pre:gfopt,sec:cr}). 

In this case, the scheme can be written in the form
\begin{subequations}\label{eq:LMS}
\begin{align}
	\frac{x_{k+1} - x_k}{h} &= -\eta_k \nabla f(x_k), \retainlabel{eq:LMS1}\\
	f(x_{k+1}) - f(x_k) &= \eta_k \inprd*{\nabla f(x_k)}{x_{k+1} - x_k}.  \retainlabel{eq:LMS2}
\end{align}
\end{subequations}
Then, $ x_{k+1} $ can be computed by solving a scalar nonlinear equation:
\[
	F_h (\eta_k; x_k) = f\paren*{x_k - \eta_k h \nabla f(x_k)} - f(x_k) +h(\eta_k)^2\norm*{\nabla f(x_k)}^2 = 0. 
\]
This equation has a trivial solution $ \eta_k = 0 $ (cf. \cref{rem:ex_gen_unique}), 
and we prove an existence theorem below for a nontrivial solution.  

\begin{theorem}\label{thm:sol1}
Let $ f : \RR^n \to \RR $ be an $L$-smooth function satisfying $ \argmin f \neq \emptyset $. 
Then, for any $x_k \in \RR^n$, there exists an $\eta_k$ that satisfies $F_h (\eta_k; x_k) = 0$ and 
\[
	\eta_k \geq \etaLB := \paren*{1 + \frac{Lh}{2}}^{-1} > 0. 
\]
\end{theorem}

\begin{proof}
In this proof, we use the notation $d_k := -\nabla f(x_k)$ for brevity. 
If $ d_k = 0 $, $F_h(\eta_k; x_k) = 0$ holds for any $\eta_k \in \RR$ so that the theorem holds. 
Therefore, we focus on the case $ d_k \neq 0 $ hereafter. 
Then, because $ \argmin f \neq \emptyset $, $f$ is bounded from below so that $ \lim_{\eta \to \infty} F_h ( \eta ; x_k ) = \infty $ holds. 

Because we assume that $f$ is $L$-smooth, the second inequality of \cref{prop:ltan} implies
\begin{align}
	F_h (\eta_k; x_k) &\leq \inprd*{\nabla f(x_k)}{\eta_khd_k} + \frac{L}{2}\norm{\eta_khd_k}^2 + h(\eta_k)^2\norm*{d_k}^2 \notag \\
	&= \eta_kh\norm*{d_k}^2 \paren*{\eta_k\paren*{1 + \frac{Lh}{2}} - 1}.  \label{eq:Fbound}
\end{align}
Therefore, $F_h \paren*{\etaLB; x_k} \le 0$ holds, 
which proves the theorem due to the intermediate value theorem. 
\end{proof}

The theorem above gives the lower bound of the nontrivial solution, 
and the following theorem gives the upper bound for a sufficiently small step size $h$. 

\begin{theorem}\label{thm:ub}
Let $ f : \RR^n \to \RR $ be an $L$-smooth function satisfying $ \argmin f \neq \emptyset $. 
Assume that $\nabla f(x_k) \neq 0$ and $h \le {2}/{L}$ hold. 
If $\eta_k > 0$ satisfies $F_h(\eta_k; x_k) = 0$, then
\[
	\etaLB \leq \eta_k \leq \paren*{1 - \frac{Lh}{2}}^{-1}
\]
holds. 
\end{theorem}

\begin{proof}
From \eqref{eq:Fbound}, $F_h (\eta_k; x_k) < 0$ holds for any $ \eta_k \in \paren*{ 0,\etaLB }$. 
Then, by using the first inequality in \cref{prop:ltan}, we see
\[
	F_h (\eta_k; x_k) \geq \eta_kh\norm*{d_k}^2 \paren*{\eta_k\paren*{1 - \frac{Lh}{2}} - 1}	
\]
(the proof is similar to the proof of \cref{thm:sol1}). 
Therefore, $F_h (\eta_k; x_k) > 0$ holds for any $\eta_k > \paren*{1 - \frac{Lh}{2}}^{-1}$, which proves the theorem. 
\end{proof}

Moreover, if $f$ is convex or satisfies the P{\L} inequality~\eqref{eq:PL}, there is an upper bound that is valid for any step size $h$. 

\begin{theorem}\label{thm:ubconv}
Let $ f : \RR^n \to \RR $ be an $L$-smooth function satisfying $ \argmin f \neq \emptyset $. 
If $f$ is convex and $\nabla f(x_k) \neq 0$ holds, 
then there exists a unique nontrivial solution $ \eta_k $ of the nonlinear equation $F_h(\eta_k; x_k) = 0$ such that $ \etaLB \le \eta_k \le 1 $ holds. 
\end{theorem}

\begin{proof}
The convexity of $f$ implies that $ F_h (\eta_k ; x_k ) $ is strictly convex with respect to $ \eta_k $ such that the nontrivial solution is unique. 

Since $f$ is convex, we see that
\begin{align}
	F_h (1; x_k) = f(x_k + hd_k) - f(x_k) + h\norm*{d_k}^2 
			  \ge \inprd*{\nabla f(x_k)}{hd_k} + h\norm*{d_k}^2 
			  = 0,
\end{align}
which proves the theorem owing to the intermediate value theorem. 
\end{proof}

\begin{theorem}\label{thm:ubPL}
Let $ f : \RR^n \to \RR $ be an $L$-smooth function satisfying $ \argmin f \neq \emptyset $. 
If $f$ satisfies the P{\L} inequality~\eqref{eq:PL} with parameter $ \mu > 0 $ and $\nabla f(x_k) \neq 0$ holds, 
there exists a nontrivial solution $ \eta_k $ of the nonlinear equation $F_h (\eta_k; x_k) = 0$ such that
	$\etaLB \leq \eta_k \leq \paren*{ 2 \mu h }^{-\frac{1}{2}}$
holds. 
\end{theorem}

\begin{proof}
By introducing $ \overline{\eta} = \paren*{ 2 \mu h }^{-\frac{1}{2}} $ and $ \overline{x} = x_k + \overline{\eta} h d_k $, we obtain
\begin{align}
    F_h (\overline{\eta}; x_k ) 
    = f(\overline{x}) - f(x_k) + \frac{1}{2 \mu } \norm*{ \nabla f (x_k
    ) }^2 
    \ge f(\overline{x}) - f(x_k) + \paren*{ f(x_k) - f^\star } 
    \ge 0,
\end{align}
which proves the theorem. 
\end{proof}


\section{Convergence rates of the Lagrange multiplier method as an optimization method}
\label{sec:cr}

The scheme~\eqref{eq:LMS} described in the previous section can be interpreted as the steepest descent method with a new step-size criterion~\eqref{eq:LMS2}. 
In this section, we show the convergence rates corresponding to \cref{thm:gradrate,thm:convrate,thm:sconvrate}. 
Hereafter, we assume $ f $ is $L$-smooth and $ 
\argmin f \neq \emptyset$ holds. 

First, we establish the discrete counterpart of \cref{thm:gradrate} as follows. 

\begin{theorem}\label{thm:LMrate1}
Let $\setE*{x_k}_{k=0}^\infty$ be a sequence satisfying \eqref{eq:LMS}, and $ \eta_k \neq 0 $ for any non-negative integer $k$. 
Then, the following inequalities hold:
\[
	\sum_{k=0}^{\infty} \norm{\nabla f(x_k)}^2 \leq \paren*{\frac{Lh}{2} + 1}^2\frac{f(x_0) - f^\star}{h},
\]
\[
	\min_{0 \leq i \leq k}\norm{\nabla f(x_i)} \leq \paren*{\frac{Lh}{2} + 1}\sqrt{\frac{f(x_0) - f^\star}{(k+1)h}}.
\] 
\end{theorem}

\begin{proof}
Similar to the proof of \cref{thm:gradrate}, we see that
\begin{align}
	f(x_0) - f^\star &\geq f(x_0) - f(x_k) 
					= -\sum_{i = 0}^{k - 1} \paren*{f(x_{i+1}) - f(x_i)}
					= h\sum_{i = 0}^{k - 1} (\eta_k)^2 \norm*{\nabla f(x_k)}^2\\
					&\geq h (\etaLB)^2 \sum_{i=0}^{k-1} \norm*{\nabla f(x_k)}^2. 
\end{align}
By the definition of $\etaLB$, the estimation above proves the theorem. 
\end{proof}

From the theorem above, we obtain the following result when $f$ is coercive. 

\begin{theorem}\label{thm:grobal}
Assume that $ f $ is coercive. 
Let $\setE*{x_k}_{k=0}^\infty$ be a sequence satisfying \eqref{eq:LMS}, and $ \eta_k \neq 0 $ for any non-negative integer $k$. 
Then, the sequence $\setE*{x_k}_{k=0}^\infty$ has an accumulation point; moreover, $ \nabla f(x^{\ast}) = 0$ holds for any accumulation point $x^{\ast}$. 
\end{theorem}

\begin{proof}
The set $\setI{ x \in \RR^n }{ f(x) \le f(x_0) }$ is compact because $f$ is coercive. 
Since the discrete dissipation law implies $\setE*{x_k}_{k=0}^\infty \subset \setI{x\in\RR^n}{f(x) \le f(x_0)}$,
$\setE*{x_k}_{k=0}^\infty$ has an accumulation point. 

For a fixed accumulation point $x^{\ast}$, there exists a convergent subsequence $\setE*{x_{k(i)}}_{i=0}^\infty$. 
The first equation of \cref{thm:LMrate1} implies that $\lim_{k \to \infty}\norm{\nabla f(x_k)} = 0$. 
Due to the continuity of the norm and $ \nabla f $, we see
\[
	\norm{\nabla f(x^{\ast})} = \norm*{\nabla f\paren*{\lim_{i \to \infty} x_{k(i)}}} = \lim_{i \to \infty} \norm{\nabla f\paren*{x_{k(i)}}} = 0, 
\]
which proves the theorem. 
\end{proof}

Moreover, if $f$ is convex or satisfies the P{\L} inequality, 
we show the discrete counterparts of \cref{thm:convrate,thm:sconvrate}. 

\begin{theorem}\label{thm:LMrate2}
Let $\setE*{x_k}_{k=0}^\infty$ be a sequence satisfying \eqref{eq:LMS}, and $ \eta_k \neq 0 $ for any non-negative integer $k$. 
If $ f $ is convex, the sequence $\setE*{x_k}_{k=0}^\infty$ satisfies
$
	f(x_k) - f^\star = O\paren*{\frac{1}{k}}
$. 
In particular, if $h \leq \frac{2}{L}$, then
\[
	f(x_k) - f^\star \leq \paren*{\frac{Lh + 2}{4}}\frac{\norm{x_0 - x^\star}^2}{kh}
\]
holds.
\end{theorem}
\begin{proof}
Let us introduce the discrete counterpart
\[
	\cE_k := \paren*{\sum_{i=0}^{k-1} h\eta_i}\paren*{f(x_k) - f^\star} + \frac{1}{2}\norm*{x_k - x^\star}^2
\]
of $ \cE $ in the proof of \cref{thm:convrate}. 
Then, we see
\begin{align}
	\mathcal{E}_{k+1} - \mathcal{E}_k &= \left(\sum_{i=0}^{k}h\eta_i\right)(f(x_{k+1}) - f(x_k)) + h\eta_k(f(x_k) - f^\star)\\
	&\qquad + \frac{1}{2}\inprd{x_{k+1} - x_k}{x_{k+1} + x_k - 2x^\star}. 
\end{align}
Here, the last term on the right-hand side can be evaluated as follows:
\begin{align}
	\frac{1}{2}\inprd{x_{k+1} - x_k}{x_{k+1} + x_k - 2x^\star} &= \frac{1}{2}\inprd{x_{k+1} - x_k}{x_{k+1} - x_k +2(x_k - x^\star)} \\
	&= \frac{1}{2}\|h\eta_k\nabla f(x_k)\|^2 + \inprd{h\eta_k \nabla f(x_k)}{x^\star - x_k} \\
	&\leq \frac{1}{2}\|h\eta_k\nabla f(x_k)\|^2 + h\eta_k(f^\star - f(x_k)). 
\end{align}
Using this evaluation, we see that
\begin{align}
	\cE_{k+1} - \cE_k &\leq \left(\sum_{i=0}^{k}h\eta_i\right)\left(-h(\eta_k)^2\|\nabla f(x_k)\|^2\right) + \frac{1}{2}\|h\eta_k\nabla f(x_k)\|^2 \\
	&= \left( \frac{1}{2} - \sum_{i=0}^{k} \eta_i \right)\|h\eta_k\nabla f(x_k)\|^2. 
\end{align}
Because $\eta_k \geq \left(\frac{Lh}{2} + 1\right)^{-1} $, there exists $k_0 \in \NN$ such that $\frac{1}{2} - \sum_{i=0}^{k} \eta_i \leq 0$ holds for any $k \geq k_0$. 
Therefore, we see that
\[
	\cE_{k_0} \geq \cE_k \geq \left(\sum_{i=0}^{k-1}h\eta_i\right)(f(x_k) - f^\star) \geq kh\left(\frac{Lh}{2} + 1\right)^{-1}(f(x_k) - f^\star),
\]
which implies $f(x_k) - f^\star = O\left(\displaystyle\frac{1}{k}\right)$. 
Moreover, if $h \leq \frac{2}{L}$, then $\eta_k \geq \paren*{\frac{Lh}{2} + 1}^{-1} \ge \frac{1}{2}$ holds. 
Since $\cE_{k+1} \leq \cE_k$ holds for any $k$, 
\[
	\frac{1}{2}\norm*{x_0 - x^\star}^2 = \cE_0 \geq \cE_k \geq kh\left(\frac{Lh}{2} + 1\right)^{-1}(f(x_k) - f^\star)
\]
holds. 
\end{proof}

\begin{theorem}\label{thm:LMrate3}
Let $\setE*{x_k}_{k=0}^\infty$ be a sequence satisfying \eqref{eq:LMS}, and $ \eta_k \neq 0 $ for any non-negative integer $k$. 
If $f$ satisfies the Polyak--{\L}ojasiewicz inequality~\eqref{eq:PL} with parameter $\mu > 0$, 
the sequence $\setE*{x_k}_{k=0}^\infty$ satisfies
\[
	f(x_k) - f^\star \leq \exp\paren*{-\frac{8\mu kh}{\paren*{Lh+2}^2}}\paren*{f(x_0) - f^\star}.
\]
\end{theorem}
\begin{proof}
We introduce the discrete counterpart $ \cL_k := f(x_k) - f^\star $ of $ \cL $ in the proof of \cref{thm:sconvrate}. 
Then, we see that
\begin{align}
	\cL_{k+1} - \cL_k &= f(x_{k+1}) - f(x_k) 
					 = -h(\eta_k)^2\norm{\nabla f(x_k)}^2 
					 \leq -2\mu h(\eta_k)^2 \paren*{f(x_k) - f^\star}\\
					 &\leq -2\mu h\paren*{\frac{Lh}{2} + 1}^{-2} \cL_{k} 
\end{align}
and
\[
	\cL_{k+1} \leq \paren*{1 - \frac{8 \mu h}{\paren*{Lh + 2}^2}} \cL_k.
\]
Because $1 + r \leq \re^r$ holds for any real number $ r $, we obtain
\[
	\cL_{k} \leq \paren*{1 - \frac{8 \mu h}{\paren*{Lh + 2}^2}}^k \cL_0 \leq \exp\paren*{- \frac{8 \mu kh}{\paren*{Lh + 2}^2}}\paren*{f(x_0) - f^\star}.
\]
\end{proof}

\section{Some relaxations of the Lagrange multiplier method}
\label{sec:proposed}

As an optimization method, the scheme~\eqref{eq:LMS} is still more expensive than the standard optimization methods. 
Therefore, in this section, 
we propose a relaxation of the scheme~\eqref{eq:LMS} that allows us to use a backtracking technique. 

In view of the dissipation law (\cref{thm:LMdisp}), 
condition~\eqref{eq:LMS2} can be relaxed to $F_h (\eta_k; x_k) \le 0 $:
we consider 
\begin{subequations}\label{eq:LMI}
\begin{align}
	\frac{x_{k+1} - x_k}{h} &= -\eta_k \nabla f(x_k), \retainlabel{eq:LMI1}\\
	f(x_{k+1}) - f(x_k) &\le \eta_k \inprd*{\nabla f(x_k)}{x_{k+1} - x_k}.  \retainlabel{eq:LMI2}
\end{align}
\end{subequations}

\begin{theorem}\label{thm:ddl}
A solution $ x_{k+1} $ of the scheme~\eqref{eq:LMI} satisfies the discrete dissipation law $ f(x_{k+1}) \le f(x_k) $
\end{theorem}

\begin{proof}
We see that
\begin{align}
    f(x_{k+1}) - f(x_k) 
    \le \eta_k \inprd*{ \nabla f (x_k) }{x_{k+1} - x_k} 
    = - h \paren*{ \eta_k }^2 \norm*{ \nabla f(x_k) }^2,
\end{align}
which proves the theorem. 
\end{proof}

Because the discrete dissipation law is crucial in the discussion in the previous section, 
we can prove the convergence rates even after this relaxation (\cref{subsec:bt}); moreover, we propose another method to adaptively change $h$ at every step, and also show convergence rates for it in \cref{subsec:adaptive}. 

\subsection{A relaxation of the Lagrange multiplier method}
\label{subsec:bt}

In this section, we consider \cref{alg:bt}. 

\begin{algorithm}[ht]
\caption{Backtracking}
\label{alg:bt}
 \begin{algorithmic}
 \Procedure{BACKTRACKING}{$x_0,h,\epsilon,\alpha$}
 \State $k \leftarrow 0$
 \While{$\norm{\nabla f(x_k)} \geq \epsilon$}
	\State $\eta_k \leftarrow 1$
 	\While {$F_h (\eta_k; x_k) > 0$}
 		\State $\eta_k \leftarrow \alpha \eta_k $
 	\EndWhile
 	\State $x_{k+1} \leftarrow x_k - h\eta_k\nabla f(x_k)$
 	\State $k \leftarrow k + 1$
 \EndWhile
 \State \textbf{return} $x_k \in \RR^n$
 \EndProcedure
 \end{algorithmic}
\end{algorithm}

Because \cref{thm:gradrate,thm:convrate,thm:sconvrate} rely on the lower bound of $ \eta_k $ as well as the discrete dissipation law, 
we establish the following lemma. 

\begin{lemma}\label{lem:lb_bt}
The iteration of backtracking in \cref{alg:bt} stops at most $ \lceil \log_{\alpha} \etaLB \rceil $ times
so that $ \eta_k \ge \alpha \etaLB $ holds. 
\end{lemma}

\begin{proof}
As shown in the proof of \cref{thm:sol1}, $ F_h ( \eta; x_k ) \le 0 $ holds for any $ \eta \le \etaLB $. 
Because $ \alpha^{ \lceil \log_{\alpha} \etaLB \rceil } \le \alpha^{ \log_{\alpha} \etaLB } = \etaLB $, 
the iteration stops at most $ \lceil \log_{\alpha} \etaLB \rceil $ times. 
Therefore, we see that $ \eta_k \ge  \alpha^{ \lceil \log_{\alpha} \etaLB \rceil } \ge \alpha^{ \log_{\alpha} \etaLB + 1 } = \alpha \etaLB $. 
\end{proof}

By using the lemma, we obtain the following convergence results. 
We omit the proof because it can be proved in a manner similar to that in \cref{thm:gradrate,thm:convrate,thm:sconvrate}. 

\begin{theorem}\label{thm:convrate_bt}
    The sequence $ \{ x_k \}_{k=0}^{\infty} $ obtained by \cref{alg:bt} satisfies
    \[ \min_{ 0 \le i \le k } \norm*{ \nabla f (x_i) } \le \frac{ L h + 2 }{2 \alpha} \sqrt{ \frac{ f(x_0) - f^\star }{ (k+1) h } }. \]
    Moreover, if $ f $ is convex, 
    \[ f ( x_k) - f^\star = O \paren*{ \frac{1}{k} } \]
    holds. 
    If $ f $ satisfies the Polyak--{\L}ojasiewicz inequality~\eqref{eq:PL} with parameter $\mu > 0$, 
    \[ f(x_k) - f^\star \le \exp \paren*{ - \frac{ 8 \alpha^2 \mu k h }{ \paren*{ L h + 2 }^2 }  } \paren*{ f(x_0) - f^\star } \]
    holds.
\end{theorem}

\subsection{Adaptive step size}
\label{subsec:adaptive}

In this section, we consider adaptively changing the step size $ h_k $ in every step. 
Here, instead of $F_h ( \eta_k ; x_k ) \le 0 $, we use the condition $ F_{h_k} (\eta_k ; x_k ) \le 0 $.
Then, $ h_{k+1} $ is defined by $ h_{k+1} = h_k \eta_k / \eta^{\ast} $, 
which is intended to maintain $ \eta_{k+1} $ around a fixed constant $ \eta^{\ast} $. 
As shown in the numerical experiments in the next section, 
this simple strategy reduces the number of backtracking iterations, 
and the numerical result does not depend significantly on the choice of $ h_0$.

\begin{algorithm}[ht]
\caption{Adaptive step size}
\label{alg:adaptive}
 \begin{algorithmic}
 \Procedure{ADAPTIVE}{$x_0, h_0 ,\epsilon, \alpha, \eta^{\ast} $} 
 \State $k \leftarrow 0$
 \While {$\norm{\nabla f(x_k)} \geq \epsilon$}
	\State $\eta_k \leftarrow 1$
 	\While {$F_{h_k} (\eta_k; x_k) > 0$}
 		\State $\eta_k \leftarrow \displaystyle\alpha \eta_k$
 	\EndWhile
 	\State $x_{k+1} \leftarrow x_k - h_k \eta_k\nabla f(x_k)$
 	\State $h_{k+1} \leftarrow \displaystyle\frac{ h_k \eta_k }{ \eta^{\ast} }$
 	\State $k \leftarrow k + 1$
 \EndWhile
 \State \textbf{return} $x_k \in \RR^n$
 \EndProcedure
 \end{algorithmic}
\end{algorithm}

In view of \cref{lem:lb_bt}, we see $ \eta_k \ge \frac{ 2 \alpha }{ L h_k + 2 } $. 
The assumption $ \eta^{\ast} < \alpha $ ensures that $\{ h_k \}_{k=0}^{\infty} $ is bounded from below, as shown in the lemma below. 
This lower bound of $\{ h_k \}_{k=0}^{\infty} $ will be used in the proof of convergence rates. 

\begin{lemma}
If $ h_0 \ge h_{\mathrm{LB}} := \frac{2 (\alpha - \eta^{\ast})}{ \eta^{\ast} L } $, then 
$ h_k \ge h_{\mathrm{LB}} $ holds for any positive integer $k$. 
\end{lemma}
\begin{proof}
We prove the lemma by induction. 
Suppose that $ h_k \ge h_{\mathrm{LB}} $ holds.
Then, we see that 
\[ h_{k+1} = \frac{ h_k \eta_k }{\eta^{\ast}} \ge \frac{ 2 \alpha h_k }{ \eta^{\ast} ( L h_k + 2) } \ge \frac{ 2 \alpha h_{\mathrm{LB}} }{ \eta^{\ast} ( L h_{\mathrm{LB}} + 2) } = h_{\mathrm{LB}}, \]
which proves the lemma.
\end{proof}

\subsubsection{Convex functions}

In this section, we deal with convex functions. 

\begin{theorem}\label{thm:LMrate_ad_conv}
We assume that $ h_0 \ge h_{\mathrm{LB}} $ and $ \eta^{\ast} \ge \frac{1}{2} $ hold. 
If $ f $ is convex, the sequence $\setE*{x_k}_{k=0}^\infty$ obtained by \cref{alg:adaptive} satisfies
\[
	f(x_k) - f^\star \leq \paren*{\frac{L}{4 \paren*{ \alpha - \eta^{\ast} } }}\frac{\norm{x_0 - x^\star}^2}{k}.
\]
\end{theorem}
\begin{proof}
Let us introduce the discrete counterpart
\[
	\cE_k := \paren*{\sum_{i=0}^{k-1} h_i \eta_i } \paren*{f(x_k) - f^\star} + \frac{1}{2}\norm*{x_k - x^\star}^2
\]
of $ \cE $ in the proof of \cref{thm:convrate}. 
Then, 
similar to the proof of \cref{thm:LMrate2}, 
we see that
\begin{equation}
	\cE_{k+1} - \cE_k \le \paren*{ \frac{1}{2} h_k - \sum_{i=0}^{k} h_i \eta_i } h_k (\eta_k)^2 \norm*{ \nabla f(x_k)}^2. 
\end{equation}
Since 
\begin{align}
    \frac{1}{2} h_k - \sum_{i=0}^{k} h_i \eta_i = \frac{1}{2} h_k - h_k \eta_k - \eta^{\ast} \sum_{i=0}^{k-1} h_{i+1} 
    = \paren*{ \frac{1}{2} - \eta_k - \eta^{\ast} } h_k - \eta^{\ast} \sum_{i=1}^{k-1} h_{i} 
    \le 0
\end{align}
holds owing to the assumption $ \eta^{\ast} \ge \frac{1}{2} $, 
we see that $ \cE_{k+1} - \cE_k \le 0 $. 
Therefore, we see that
\[
	\frac{1}{2}\norm*{x_0 - x^\star}^2 = \cE_{0} \geq \cE_k \geq \left(\sum_{i=0}^{k-1}h_i\eta_i\right)(f(x_k) - f^\star) \geq k \eta^{\ast} h_{\mathrm{LB}} (f(x_k) - f^\star),
\]
which proves the theorem. 
\end{proof}

\subsubsection{Functions satisfying P{\L} inequality}

In this section, we deal with functions satisfying P{\L} inequality. 
In this case, we need the upper bound of $ \{ h_k \}_{k=0}^{\infty} $ as well as the lower bound. 

\begin{lemma}
Assume that $f$ satisfies the Polyak--{\L}ojasiewicz inequality~\eqref{eq:PL} with parameter $\mu > 0$, and $ h_0 \le h_{\mathrm{UB}} := \frac{1}{ 2 \mu (\eta^{\ast})^2 } $ holds. Then, $ h_k \le h_{\mathrm{UB}} $ holds for any positive integer $k$.  
\end{lemma}

\begin{proof}
    In a manner similar to the proof of \cref{thm:ubPL}, we see that $ \eta_k \le \sqrt{\frac{1}{2 \mu h_k}} $. 
    Then, we prove the lemma by induction. 
    Suppose that $ h_k \le h_{\mathrm{UB}}$ holds. 
    Then, we see that
    \[ h_{k+1} = \frac{ h_k \eta_k }{\eta^{\ast}} \le \frac{1 }{\eta^{\ast}} \sqrt{ \frac{ h_k }{ 2 \mu } } \le \frac{1 }{\eta^{\ast}} \sqrt{ \frac{ h_{\mathrm{UB}} }{ 2 \mu } } = h_{\mathrm{UB}}, \]
    which proves the lemma. 
\end{proof}

\begin{theorem}\label{thm:LMrate_ad_PL}
Assume that $ h_0 \in [h_{\mathrm{LB}},h_{\mathrm{UB}}] $ holds. 
If $f$ satisfies the Polyak--{\L}ojasiewicz inequality~\eqref{eq:PL} with parameter $\mu > 0$, 
the sequence $\setE*{x_k}_{k=0}^\infty$ obtained by \cref{alg:adaptive} satisfies
\[
	f(x_k) - f^\star \leq \exp\paren*{- \frac{16 \alpha (\alpha - \eta^{\ast}) (\eta^{\ast})^2 }{ \kappa \paren*{ \kappa + 4  (\eta^{\ast})^2 } } k } \paren*{f(x_0) - f^\star},
\]
where $ \kappa := L / \mu $ is the condition number. 
\end{theorem}
\begin{proof}
We introduce the discrete counterpart $ \cL_k := f(x_k) - f^\star $ of $ \cL $ in the proof of \cref{thm:sconvrate}. 
Then, we see that
\begin{align}
	\cL_{k+1} - \cL_k = f(x_{k+1}) - f(x_k) 
					 \le -h_k (\eta_k)^2\norm{\nabla f(x_k)}^2 
					 \leq -2\mu h_k (\eta_k)^2 \paren*{f(x_k) - f^\star}.		
\end{align}
By using $ h_k (\eta_k)^2 = \eta^{\ast} h_{k+1} \eta_k \ge \eta^{\ast} h_{\mathrm{LB}} \frac{2\alpha}{Lh_{\mathrm{UB}}+2} = \frac{8 \alpha \mu (\alpha - \eta^{\ast}) (\eta^{\ast})^2 }{ L \paren*{ L + 4 \mu (\eta^{\ast})^2 } } $, we obtain
\[
	\cL_{k+1} \leq \paren*{1 - \frac{16 \alpha \mu^2 (\alpha - \eta^{\ast}) (\eta^{\ast})^2 }{ L \paren*{ L + 4 \mu (\eta^{\ast})^2 } } } \cL_k, 
\]
which proves the theorem. 
\end{proof}

\section{Numerical experiments}
\label{sec:ne}

The efficacy of the Lagrange multiplier method for differential equations is well described by Cheng, Liu, and Shen~\cite{CLSJ2020}. 
Therefore, in this section, we focus on the application for optimization: 
we compare \cref{alg:bt,alg:adaptive} with the steepest descent method with a fixed step size $ h = 1/L $ and the step size satisfying the standard Armijo rule:
\[ f(x_k - h_k \nabla f(x_k)) - f(x_k) \le - c h_k \norm*{  \nabla f(x_k) }^2, \]
where $ c \in (0,1) $ is a parameter, and $h_k$ is obtained by a standard backtracking line search with the parameter $ \alpha \in (0,1)$. 

Throughout the numerical experiment in this section,
the parameter $ \alpha $ in the Armijo rule and \cref{alg:bt,alg:adaptive} is fixed at $ \alpha = 0.8 $.
Because we investigate the difference in the results depending on the step size criteria in this experiment, we choose the parameter $ \alpha $ corresponding to a relatively precise line search.
In addition, we fix the parameter $ \eta^{\ast} = 0.5 $ in view of \cref{thm:LMrate_ad_conv}. 

\subsection{Quadratic function}
\label{subsec:quad}

First, we consider the quadratic function
\begin{equation}\label{eq:prob_quad}
    f(x) = \frac{1}{2} \inprd{x}{Ax} + \inprd{b}{x},
\end{equation}
where $ A \in \RR^{n \times n} $ and $ b \in \RR^{n} $. 
In this section, we fix $ n = 500 $ and $ b \in \RR^{n} $, whose elements are independently sampled from the normal distribution $ \mathcal{N} (0,5) $. 
We also fix the symmetric positive definite matrix $A \in \RR^{ n \times n }$, defined as $ A = Q^{\top} \Lambda Q $ by using a diagonal matrix $ \Lambda $, whose elements are sampled from a uniform distribution on $[0.001, 1]$, 
and an orthogonal matrix $Q$ that was sampled from the Haar measure on the orthogonal group.
The resulting matrix $A$ has the maximum eigenvalue of $ L \approx 0.998 $ and minimum eigenvalue of $\mu \approx 0.0022$. 
We set the initial step size of the backtracking line search for the Armijo rule to $ 10 $.

\begin{figure}[ht]
\includegraphics{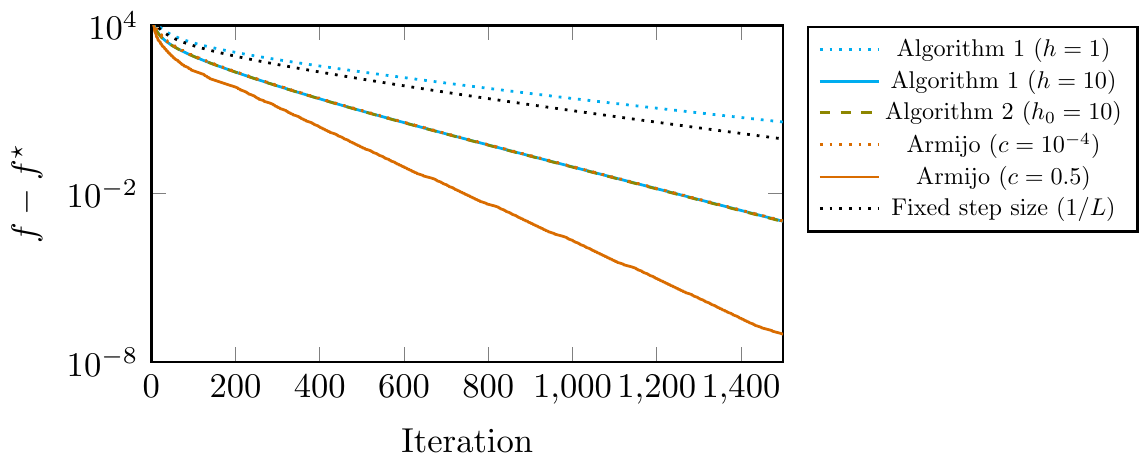}
    \caption{Evolution of function values for the quadratic function~\eqref{eq:prob_quad}.}
    \label{fig:quad}
\end{figure}

\begin{table}[ht]
\centering
\caption{Average step size and number of backtracking iterations for the quadratic function~\eqref{eq:prob_quad}. }
\label{tab:quad}
\begin{tabular}{c||ccc|ccc|ccc}\hline
    Method     & \multicolumn{3}{c|}{Armijo ($c$)} & \multicolumn{3}{c|}{\cref{alg:bt} ($h$)} & \multicolumn{3}{c}{\cref{alg:adaptive} ($h_0$)} \\ \hline
    Parameter  & $10^{-4}$ & $0.1$   & $ 0.5 $     & $1$   & $10$    & $100$           & $1$     & $10$    & $100$ \\ \hline \hline
    step size  & $2.017$   & $2.016$ & $3.398$     & $0.8$ & $2.016$ & $2.024$         & $2.022$ & $2.020$ & $2.020$ \\
    \# iterations & $7.19$  & $7.19$  & $6.16$      & $1$   & $7.19$  & $17.51$         & $3.10$  & $3.11$  & $3.12$ \\ \hline
\end{tabular}
\end{table}

\Cref{fig:quad} summarizes the evolution of function values and \cref{tab:quad} summarizes the average step size and the number of backtracking iterations. 
In \cref{fig:quad}, we omit the Armijo rule with $ c = 0.1 $, \cref{alg:bt} with $h=100$, and \cref{alg:adaptive} with $h_0 = 1, 100$ because they are very similar to the Armijo rule with $ c = 10^{-4} $, \cref{alg:bt} with $h=10$, and \cref{alg:adaptive} with $h_0 = 10$, respectively. 

The results of \Cref{alg:bt} with an appropriate $h$ and \cref{alg:adaptive} are similar to those of the Armijo rule with a small $c$. 
Because the Armijo rule with $c = 0.5$ is similar to the exact line search for quadratic functions, it overwhelms the other methods.

\subsection{Log-Sum-Exp function}
\label{subsec:lse}

Second, we consider the Log-Sum-Exp function:
\begin{equation}\label{eq:prob_lse}
    f(x) = \rho \log \paren*{ \sum_{i=1}^m \exp \paren*{ \frac{ \inprd{a_i}{x} - b_i }{\rho} } }. 
\end{equation}
where $ a_i \in \RR^n \ (1 \le i \le m)$, $ b_i \in \RR \ (1 \le i \le m) $ and $ \rho > 0 $. 
In this section, we fix $ n = 50 $, $m = 200 $, and $ \rho = 20 $. 
We also fix $ a_i $ and $ b_i $, whose elements are independently sampled from the normal distribution $ \mathcal{N} (0,1) $ and $ \mathcal{N} (0,\sqrt{2}) $, respectively. 
The resulting $ a_i $ satisfies $ \max_{ 1 \le k \le m } \norm{a_k}^2 \approx 42.687 $, and 
the Lipschitz constant $L$ satisfies $ L \le \max_{ 1 \le k \le m } \norm{a_k}^2 $. 
We set the initial step size of the backtracking line search for the Armijo rule to $ 100 $.

\begin{figure}[ht]
\includegraphics{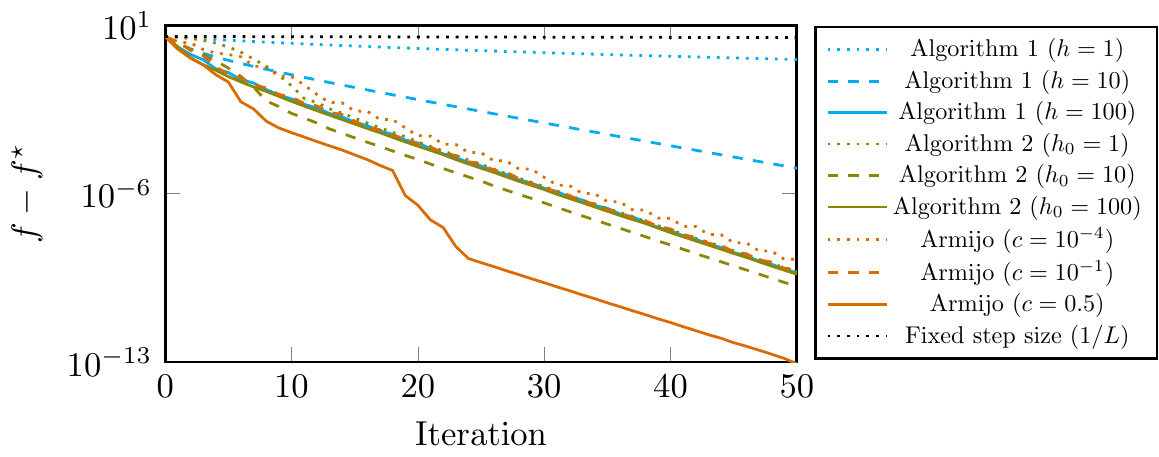}
    \caption{Evolution of function values for the Log-Sum-Exp function~\eqref{eq:prob_lse}.}
    \label{fig:lse}
\end{figure}

\begin{table}[ht]
\centering
\caption{Average step size and number of backtracking iterations for the Log-Sum-Exp function~\eqref{eq:prob_lse}. }
\label{tab:lse}
\begin{tabular}{c||ccc|ccc|ccc}\hline
    Method     & \multicolumn{3}{c|}{Armijo ($c$)} & \multicolumn{3}{c|}{\cref{alg:bt} ($h$)} & \multicolumn{3}{c}{\cref{alg:adaptive} ($h_0$)} \\ \hline
    Parameter  & $10^{-4}$ & $0.1$    & $ 0.5 $    & $1$   & $10$    & $100$           & $1$     & $10$    & $100$ \\ \hline \hline
    step size  & $15.44$   & $15.29$  & $18.80$    & $0.8$ & $7.97$  & $15.18$         & $14.66$ & $15.67$ & $15.07$ \\
    \# iterations & $8.42$  & $8.46$   & $8.14$     & $1$   & $1.02$  & $8.48$          & $2.80$  & $3.02$  & $3.22$ \\ \hline
\end{tabular}
\end{table}

\Cref{fig:lse} summarizes the evolution of function values and \cref{tab:lse} summarizes the average step size and the number of backtracking iterations. 
The results of \cref{alg:bt} with an appropriate $h$ and \cref{alg:adaptive} are similar to those of the Armijo rule with a small $c$. 
Although the Armijo rule with $c = 0.5$ converges faster than the other methods, 
the rate itself is similar to \cref{alg:bt} with an appropriate $h$ and \cref{alg:adaptive}.

\subsection{A nonconvex function satisfying P{\L} inequality}
\label{subsec:noncon}

Finally, we consider the function
\begin{equation}\label{eq:prob_noncon}
    f(x) = \norm*{x}^2 + 3 \sin^2 \paren*{ \inprd{b}{x} }
\end{equation}
used in \cite{ERRS2018}, where 
$ b \in \RR^n $ is a vector that satisfies $ \norm*{b} = 1 $. 
This function is $8$-smooth, nonconvex, and satisfies the Polyak--{\L}ojasiewicz inequality~\ref{eq:PL} with parameter $\mu = 1/32$. 
In this section, we fix $ n = 50 $ and $ b = v / \norm*{v} \in \RR^{n} $, where the elements of $v \in \RR^n$ are independently sampled from the normal distribution $ \mathcal{N} (0,1) $. 
We set the initial step size of the backtracking line search for the Armijo rule to $ 10 $.

\begin{figure}[ht]
\includegraphics{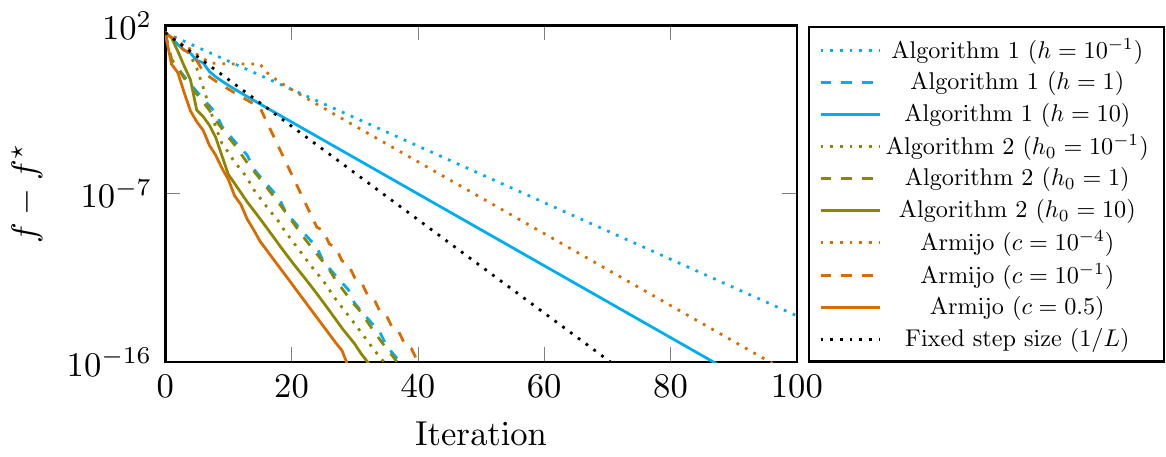}
    \caption{Evolution of function values for the nonconvex function~\eqref{eq:prob_noncon}.}
    \label{fig:noncon}
\end{figure}

\begin{table}[ht]
\centering
\caption{Average step size and number of backtracking iterations for the nonconvex function~\eqref{eq:prob_noncon}. }
\label{tab:noncon}
\begin{tabular}{c||ccc|ccc|ccc}\hline
    Method     & \multicolumn{3}{c|}{Armijo ($c$)} & \multicolumn{3}{c|}{\cref{alg:bt} ($h$)} & \multicolumn{3}{c}{\cref{alg:adaptive} ($h_0$)} \\ \hline
    Parameter  & $10^{-4}$ & $0.1$    & $ 0.5 $    & $0.1$  & $1$     & $10$            & $1$     & $10$    & $100$ \\ \hline \hline
    step size  & $0.260$   & $0.227$  & $0.235$    & $0.08$ & $0.205$ & $0.249$         & $0.207$ & $0.204$ & $0.215$ \\
    \# iterations & $16.6$  & $17.2$   & $17.2$     & $1$    & $7.1$   & $16.8$          & $3.04$  & $3.15$  & $3.26$ \\ \hline
\end{tabular}
\end{table}

\Cref{fig:noncon} summarizes the evolution of function values and \cref{tab:noncon} summarizes the average step size and the number of backtracking iterations. 
The results of \cref{alg:bt} with an appropriate $h$ and \cref{alg:adaptive} are similar to those of the Armijo rule.

\section{Conclusion}
\label{sec:conclusion}

In this paper, we established existence results on the Lagrange multiplier approach, a recent geometric numerical integration technique, for the gradient system. 
In addition, we showed that, when $ Q $ is the zero matrix, the Lagrange multiplier approach reads a new step-size criterion for the steepest descent method. 
Thanks to the discrete dissipation law, 
the convergence rates of the proposed method for several cases can be proved in a form similar to the discussions on ODEs. 
In this paper, we focused only on the simplest gradient flow, 
but the results suggest that geometric numerical integration techniques can be effective for other ODEs appearing in optimization problems. 

Several issues remain to be investigated. 
First, it would be interesting to investigate the application of geometric numerical integration techniques to other ODEs that appear during optimization. 
Second, the existence results in this paper are only for a special case of the Lagrange multiplier approach. 
Because the assumption $ x_{k+1/2}^{\ast} := x_k $ is a bit restrictive in the usual numerical integration of ODEs and PDEs, 
it is important to generalize the existence results.

\section*{Acknowledgements}

The authors are grateful to Takayasu Matsuo and Naoki Marumo for their valuable comments. 

\appendix

\section{An extension of \cref{existence:special}}
\label{app:ex}

In this section, we consider the scheme~\eqref{eq:LM} with the assumption 
$ x^{\ast}_{k+1/2} = x_k $ and $ Q = 0 $ 
(note that we further assume $ \pdmat = I $ in \cref{existence:special}). 
In this case, the scheme can be written as
\begin{subequations}\label{eq:LMS_A}
\begin{align}
	\frac{x_{k+1} - x_k}{h} &= -\eta_k \pdmat \nabla f(x_k), \label{eq:LMS1_A}\\
	f(x_{k+1}) - f(x_k) &= \eta_k \inprd*{\nabla f(x_k)}{x_{k+1} - x_k}.  \label{eq:LMS2_A}
\end{align}
\end{subequations}
Then, $ x_{k+1} $ can be computed by solving a scalar nonlinear equation
\[
	F_h (\eta_k; x_k) = f\paren*{x_k - \eta_k h \pdmat \nabla f(x_k)} - f(x_k) +h(\eta_k)^2 \inprd*{ \nabla f(x_k) }{ \pdmat \nabla f (x_k) } = 0. 
\]

Even in this case, the counterparts of \cref{thm:sol1,thm:ub,thm:ubconv,thm:ubPL} hold as follows. 
We omit their proofs because they are similar to those of the counterparts in \cref{existence:special}. 

\begin{theorem}\label{thm:sol1_A}
For any $x_k \in \RR^n$, there exists an $\eta_k$ that satisfies $F_h (\eta_k; x_k) = 0$ and 
\[
	\eta_k \geq \paren*{1 + \frac{Lh}{2 \rqmin \paren*{ \pdmat^{-1} } }}^{-1} > 0.
\]
\end{theorem}

\begin{theorem}\label{thm:ub_A}
Assume that $\nabla f(x_k) \neq 0$ and $h \le {2 \rqmin \paren*{ \pdmat^{-1} } }/{L}$ hold. 
If $\eta_k > 0$ satisfies $F_h (\eta_k; x_k) = 0$, then
\[
	\paren*{1 + \frac{Lh}{2\rqmin \paren*{ \pdmat^{-1} }}}^{-1} \leq \eta_k \leq \paren*{1 - \frac{Lh}{2\rqmin \paren*{ \pdmat^{-1} }}}^{-1}
\]
holds. 
\end{theorem}



\begin{theorem}\label{thm:ubconv_A}
If $f$ is a convex function and $\nabla f(x_k) \neq 0$ holds, 
there exists a unique nontrivial solution $ \eta_k $ of the nonlinear equation $F_h (\eta_k; x_k) = 0$ satisfying
\[
	\paren*{1 + \frac{Lh}{2\rqmin \paren*{ \pdmat^{-1} }}}^{-1} \leq \eta_k \leq 1.
\]
\end{theorem}

\begin{theorem}\label{thm:ubPL_A}
If $f$ satisfies the P{\L} inequality~\eqref{eq:PL} with parameter $ \mu > 0 $ and $\nabla f(x_k) \neq 0$ holds, 
there exists a nontrivial solution $ \eta_k $ of the nonlinear equation $F_h (\eta_k; x_k) = 0$ satisfying
\[
	\paren*{1 + \frac{Lh}{2\rqmin \paren*{ \pdmat^{-1} }}}^{-1} \leq \eta_k \leq \paren*{ 2 \mu h \rqmin \paren*{ \pdmat^{-1} } }^{-\frac{1}{2}}.
\]
\end{theorem}

\bibliographystyle{abbrv}
\bibliography{reference}

\end{document}